\documentclass[oneside]{scrartcl}
\usepackage[utf8]{inputenc}\usepackage{lmodern}
\usepackage[T1]{fontenc}
\usepackage{verbatim}
\usepackage{enumitem}
\usepackage[sans]{dsfont}
\usepackage{amssymb,amsbsy,bbm,fontawesome,nicefrac,tikzsymbols,hyperref,bclogo,multirow,yhmath,bbding,mathtools}
\usepackage{soul}

\usepackage{amsmath,amssymb,amsthm,mathrsfs, cancel}
\usepackage{natbib}

\makeatletter
\theoremstyle{plain}

\newtheorem{proposition}{Proposition} 
\newtheorem{defn}{Definition}

\newtheorem{thm}{Theorem}
\newtheorem{lemma}{Lemma}

\theoremstyle{remark}
\newtheorem{remark}{Remark}

\newcommand{\bet}{\zeta}
\newcommand{\yy}{\widehat{y^\circ}_T}
\newcommand{\y}{\widehat{y^\ast}_T}

\newcommand{\lebesgue}{\lambda}
\def\C{\mathbf{C}}

\def\E{\mathbb{E}}

\def\F{\mathscr{F}}

\def\N{\mathbb{N}}

\def\R{\mathbb{R}}

\def\rd{\mathbbm{R}}

\def\Ma{\mathbbm{M}}

\def\ML{\mathcal{L}}

\def\N{\mathbb{N}}

\def\X{X}

\newcommand{\yn}{0}
\newcommand{\com}{D}

\renewcommand{\d}{\mathrm{d}}

\newcommand{\e}{\mathrm{e}}

\newcommand{\LL}{\mathcal{L}}
\newcommand{\ep}{\varepsilon}

\renewcommand{\hat}{\widehat}
\renewcommand{\Upsilon}{Z}
\renewcommand{\tilde}{\widetilde}%

\definecolor{cs}{rgb}{0,0,0}
\definecolor{darkspringgreen}{rgb}{0.09, 0.45, 0.27}
\begin{document}
	\title{Nonparametric learning for impulse control problems}
	\subtitle{Exploration vs. Exploitation}
	\author{S\"oren Christensen\thanks{Christian-Albrechts-Universität Kiel, Mathematisches Seminar, Heinrich-Hecht-Platz 6, 24118 Kiel, Germany} \and Claudia Strauch\thanks{Aarhus University, Department of Mathematics, Ny Munkegade 118, 8000 Aarhus C, Denmark} }
	
	\maketitle

%
%
%
%
%
%


\begin{abstract}
One of the fundamental assumptions in stochastic control of continuous time processes is that the dynamics of the underlying (diffusion) process is known. 
This is, however, usually obviously not fulfilled in practice. 
On the other hand, over the last decades, a rich theory for nonparametric estimation of the drift (and volatility) for continuous time processes has been developed.
The aim of this paper is bringing together techniques from stochastic control with methods from statistics for stochastic processes to find a way to both learn the dynamics of the underlying process and control in a reasonable way at the same time. 
More precisely, we study a long-term average impulse control problem, a stochastic version of the classical Faustmann timber harvesting problem. 
One of the problems that immediately arises is an exploration-exploitation dilemma as is well known for problems in machine learning. 
We propose a way to deal with this issue by combining exploration and exploitation periods in a suitable way. 
Our main finding is that this construction can be based on the rates of convergence of estimators for the invariant density.
Using this, we obtain that the average cumulated regret is of uniform order $O({T^{-1/3}})$.
\end{abstract}


\vspace{.2cm}

\section{Introduction}
\subsection{Related results and literature}
In the usual stochastic control setup, it is assumed that the dynamics of the underlying stochastic process is known to the decision maker. 
In practice, this assumption does not seem to be realistic in many situations of interest. 
Different approaches have been proposed to deal with this issue. 
One way is to take the point of view of an ambiguity averse decision maker who is uncertain about the drift of the underlying process. 
She then searches for a maximin rule to find the best decision given nature chooses the worst possible drift. 
This approach is referred to as Knightian uncertainty, see \cite{R}, \cite{christensen2013optimal} and \cite{bayraktar2015minimizing}. 
({For these and the following references, we have, from the large body of literature, tried to choose references which are close to the setting studied in this paper.}) 
This is a conservative point of view for the decision maker. 
In particular, she does not take into account the information obtained about the drift by observing the process over time. 

Another perspective is taken by using hidden Markov models for modeling partially observed systems. 
For examples, it can be assumed that the coefficients of the underlying system are driven by an unobservable stochastic process which is then usually treated by filtering techniques. 
The problem formulation can then be done from a Bayesian or frequentist perspective. 
See, e.g., \cite{rieder2005portfolio} for an example from portfolio optimization. 
Another strand of literature with this point of view is classical sequential statistics, in particular optimal estimation of the drift and related quantities, see, e.g., \cite{ps} for different examples, or \cite{jobjornsson2018anscombe} for an overview on applications in clinical trial design. 
Here, for tractable examples a low-dimensional parametric situation is assumed. 
The same holds true for a strand of literature referred to as adaptive control theory, as considered 
in \cite{stettner1986continuous,bielecki2019adaptive}. 

From a different perspective, the problem is treated in the area of machine learning. 
In the framework of a Markov decision process, it is assumed that the decision maker has no knowledge about both the dynamics of the underlying system and the reward structure. 
Extensive textbook treatments with many references are \cite{MR3642732,bertsekas_reinf}. 
A main challenge is the exploration vs.~exploitation trade-off: on the one hand the decision maker needs to learn the environment, but at the same time she wants to use her current knowledge to optimize her reward. 
In many situations, there is no hope to find an optimal (in any reasonable sense) strategy. 
A deeper mathematical treatment, in particular for large (or even infinite) underlying models, is usually not obtainable. 
The exploration vs.~exploitation trade-off is also well-known from the famous multi-armed bandit problem.
We refer to \cite{bubeck2012regret} and \cite{Lattimore_bandits}.

In principle, the question of exploration (i.e., learning the dynamics of the underlying process) is a statistical problem. 
{In cases where application-driven knowledge of the form of the dynamics is available, one might try to exploit this by postulating a parametric diffusion model, indexed by some finite-dimensional parameter $\theta$. 
In most situations however, it is not plausible to assume a low-dimensional parametric structure, and one has to resort to nonparametric methods from the field of statistics for stochastic processes.
The statistical nature of the problem of estimating the diffusion dynamics depends heavily on the dimension of the diffusion process and the assumption on the observation scheme.
In particular, the analysis and the statistical results change when passing from the scalar case to higher dimension and from the framework of continuous or high-frequency observations to low-frequency sampling.
A summary on the evolution of the area of statistical estimation for diffusions up to the mid 2000's is given in \cite{goetal04}. 
The monograph \cite{kut04} provides a comprehensive overview on inference for one-dimensional ergodic diffusion processes on the basis of continuous observations considering pointwise
and $L^2$-risk measures.}

In the light of the previous discussion, it is a bit surprising that the application of results from nonparametric statistics to problems in stochastic control is rare. 
One exception is \cite{kohler2013data} where a purely data driven method for tackling discrete optimal stopping problems is proposed. 
Here, the exploration and exploitation periods have been separated by assuming that first the underlying process has been observed for $T$ time units before the stopping problem is considered in a second step on the basis of the observations. 
The authors then prove convergence to the optimal value as $T\to\infty$ when using their rule for general underlying stationary and ergodic processes in discrete time. 
No results on the speed of convergence are obtained. 

\subsection{Problem studied in this paper and heuristic discussion}
In this paper, we consider a continuous time impulse control problem for an underlying one-dimensional diffusion process. 
We relax the assumption of known drift in the most fundamental way by assuming no prior knowledge about the drift function for the underlying diffusion when the optimization problem starts. 
We, however, assume that we know the reward structure of the problem. 
Our aim is then to find a purely data driven procedure giving, on the long run, the optimal reward per time unit. 
Even more, we aim in obtaining a (preferably high) rate of convergence by taking into account the known rates of convergence in the estimation of the {unknown diffusion dynamics. 
Given the continuous nature of the control problem, it suggests itself to work in the statistical framework of continuous observations of the diffusion process.
In this case, the diffusion coefficient $\sigma$ is perfectly identified from the data by means of the quadratic variation of $X$.}

We consider one of the best known stochastic impulse control problems in natural resource economics: the stochastic Faustmann harvesting problem. 
We refer to \cite{A04} for an overview with many references. 
For our setting, the problem may briefly be described as follows (an exact formulation including all assumptions is given in the following sections): 
we denote by $X$ the random volume of a natural resource, say, standardized to be a process on the full real line. For historical reasons, $X$ is often interpreted as volume of timber in a forest stand. Although this interpretation is not totally convincing in this stochastic setting, we will also use these notions in the following. 

In case no interventions take place, the process evolves as an It\^o diffusion process with dynamics 
\begin{equation*}
\d X_t=b(X_t)\d t+\sigma(X_t)\d W_t.
\end{equation*}
The decision maker can now choose intervention times $0 \leq \tau_1<\tau_2<\ldots$ where the forest is harvested to an externally given base level $y_0\in \R$, i.e., $X_{\tau_n}=y_0$ for all\ $n$. 
Selling the wood yields the gross reward $g(X_{\tau_n-})-g(y_0+)$. 
The net reward is given by $$\left(g(X_{\tau_n-})-g(y_0+)\right)+g(y_0+)=g(X_{\tau_n-}),$$ where $g(y_0+)\leq 0$ can be interpreted as a fixed cost component -- the costs of one intervention.
W.l.o.g. we assume $y_0=\yn$ in the following. 
The decision maker aims in choosing the intervention times to maximize the asymptotic growth rate 
\begin{equation*}
\liminf_{T\to\infty}\frac{1}{T}\ \E\left[\sum_{n:\tau_n\leq T}g(X_{\tau_n-})\right].
\end{equation*}
This long-term average criterion reflects a sustainable approach to the harvesting problem. 
In the existing literature, variants of this problem (mostly with a discounted reward criterion) are treated under the assumption that the dynamics of the underlying process $X$ is fully known to the decision maker. 
In this case, a stationary optimal strategy may be found in the class of threshold strategies of the form
\[\tau_n=\inf\left\{t\geq \tau_{n-1}:X_t\geq y^*\right\},\] 
see Figure \ref{fig:plotssstrategy}. 
As we detail in Section \ref{sec:known_b} below, for known dynamics $b(\cdot)$, $\sigma(\cdot)$, the threshold $y^*>\yn$ can be characterized as a maximum of an explicitly given real function.

\begin{figure}
	\centering
	\includegraphics[width=0.9\linewidth]{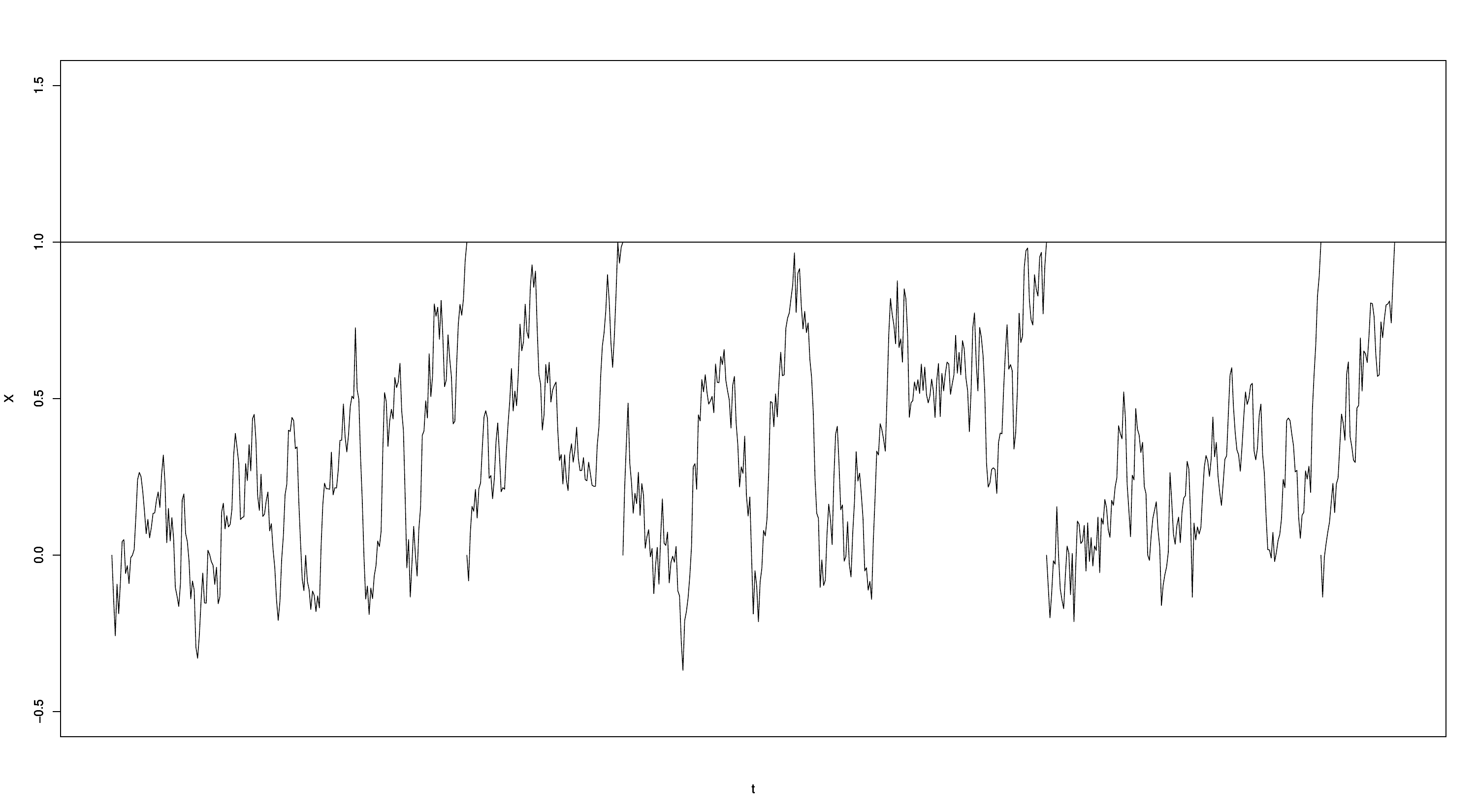}
	\caption{A path controlled using a threshold strategy}
	\label{fig:plotssstrategy}
\end{figure}

As discussed above, however, the assumption that the drift function $b(\cdot)$ is known, are not realistic in practice. 
Therefore, the threshold $y^*=y^*_b$ is not known as well and has to be estimated. 
A naive idea could be as follows: first, we let the forest grow without intervention and use methods from statistics for diffusion processes {to estimate the unknown dynamics of the diffusion process} (or, more precisely, the invariant density). 
Using this estimator, we then estimate the optimal threshold $\hat y$ and use the corresponding threshold strategy. 
Then, we update the estimator when observing the process, see Figure \ref{fig:plot_naive}.
However, using this strategy we are faced with an exploration vs.~exploitation problem: 
we cannot learn the {invariant density associated with the unknown drift coefficient $b(\cdot)$} above the level $\hat y$ as we always cut the forest when the value $\hat y$ is reached. 
So, it might be that the forest stand grows very rapidly above $\hat y$, so that we should wait longer, but we never learn this as the process never reaches this region again. 
\begin{figure}
	\centering
	\includegraphics[width=0.9\linewidth]{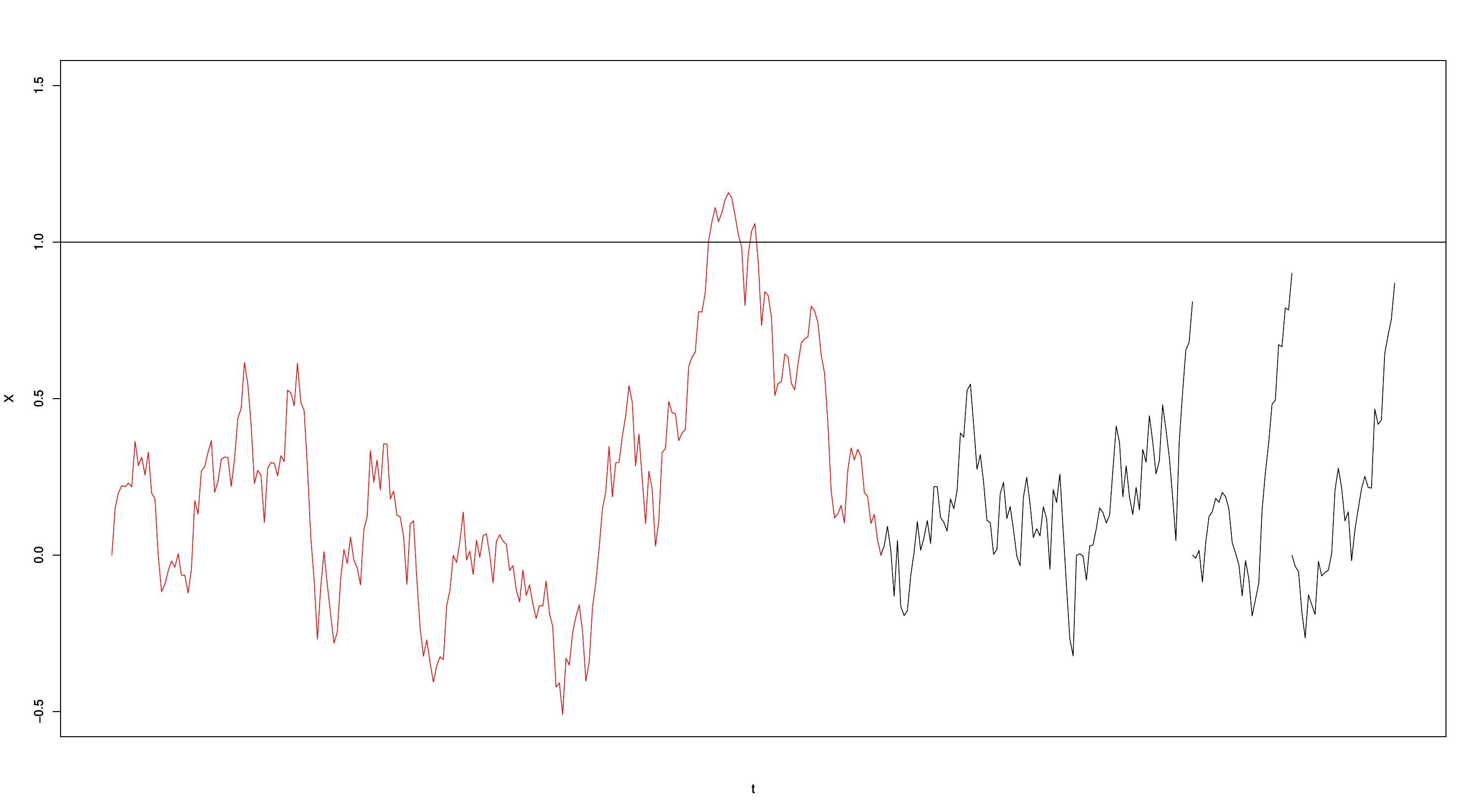}
	\caption{A path controlled using a naive impulse strategy with an exploration period at the beginning (red) and then a greedy strategy (black)}
	\label{fig:plot_naive}
\end{figure}

The way out we take here is to modify the strategy such that we combine exploitation with exploration periods: from time to time, we ignore the estimated optimal level $\hat y$ and let the forest grow further to obtain information about the drift function. 
To still have a chance to obtain the optimal growth rate from the full information problem in the long run, it is clear that the time $S_T$ spent in the exploration periods until $T$ should fulfill
\[S_T\to\infty,\;\;\frac{S_T}{T}\to0\;\;\;\mbox{ as }T\to\infty. \]
However, a suitable rate for this is not clear at first. 
{Our investigation reveals that the fundamental statistical issue for proposing a data-driven strategy for the given impulse control problem consists in bounding the $L^1$ risk of an estimator of the invariant density of the underlying diffusion process.
The convergence rate of this estimator can then be used to derive an appropriate rate for balancing exploitation and exploration periods.}
More precisely, we obtain that a rate of $S_T\approx T^{2/3}$ is a suitable choice. 

Recently, certain singular control problems are studied in \cite{christensen2021learning} using an approach inspired by the one presented here.

\subsection{Structure of this paper}
More detailed, we start by precisely stating the problem in Section \ref{sec:problem}. 
Then, we analyze the problem in two steps: in Section \ref{sec:known_b}, we first solve the ergodic stochastic control problem for known characteristics of the underlying process. Motivated by this, we study the statistical problem of estimating the main ingredients for the optimal strategy in Section \ref{sec:rhobest}. 
The analysis is based on the analysis of the $L^1$ risk of invariant density estimators. 
In\ Section \ref{sec:data_driven}, we then put pieces together as indicated above.

\section{Problem formulation and standing assumptions}\label{sec:problem}
We consider an It\^o diffusion process $X$ on the real line with dynamics
\begin{equation}\label{eq:dynamcs}
\d X_t=b(X_t)\d t+\sigma(X_t)\d W_t.
\end{equation}
{For ease of presentation, we develop our results in the following classical scalar diffusion model. 
As discussed above, we may assume $\sigma$ to be known. 
Later, we will consider the diffusion for varying drift functions $b$. 
If necessary, we denote this by writing $\E_b$ for the expectation with underlying diffusion with drift function $b$. 
	
\begin{defn}\label{def:B}
{Let $\sigma\in \operatorname{Lip}_{\operatorname{loc}}(\R)$ and assume that, for some constants $\overline{\nu},\underline\nu\in(0,\infty),$  $\sigma$ satisfies $\underline\nu\leq \left|\sigma(x)\right|\leq \overline\nu$ for all $x\in\R$. 
For fixed constants $A,\gamma>0$ and $\C\geq 1$, define the set $\Sigma=\Sigma(\C,A,\gamma,\sigma)$ as
\[\Sigma\coloneqq\Big\{b \in \operatorname{Lip}_{\operatorname{loc}}(\R)\colon|b(x)| \leq\C(1+|x|),\ \forall|x|>A\colon \frac{b(x)}{\sigma^2(x)}\operatorname{sgn}(x)\leq -\gamma\Big\}.
\]}
\end{defn}
	
Given $\sigma$ as above and any $b\in\Sigma$, there exists a unique strong solution of the SDE \eqref{eq:dynamcs} with ergodic properties and invariant density
\[	\rho(x)=\rho_b(x)\coloneqq \frac{1}{C_{b,\sigma}\sigma^2(x)}\exp\left(\int_0^x\frac{2b(y)}{\sigma^2(y)}\d y\right),\quad x\in\R,
\]
with
	\[C_{b,\sigma}\coloneqq\int_\R\frac{1}{\sigma^2(u)}\ \exp\left(\int_0^u\frac{2b(y)}{\sigma^2(y)}\d y\right)\d u\] 
denoting the normalizing constant, see \cite{borodin-salminen}, as is needed for obtaining the relevant statistical results.
Note that the assumptions on the drift coefficient already impose some regularity on the invariant density $\rho_b$. 
More precisely, for any $b\in\Sigma$, $\sigma^2\rho_b$ is continuously differentiable and there exists a constant $\ML >0$ (depending only on $\C,A,\gamma$) such that 
\begin{equation}\label{eq:reg_rho_b}
	\sup_{b\in{\color{cs}\Sigma(\C,A,\gamma,\sigma)}}\max\left\{\|\rho_b\|_{\infty},\ \|(\sigma^2\rho_b)'\|_\infty\right\}< \ML 
\end{equation}
and, for any $\theta>0$, we have $\sup_{b\in{\color{cs}\Sigma(\C,A,\gamma,\sigma)}}\sup_{x\in\R}\left\{|x|^{\theta}\rho_b(x)\right\}<\infty$.
}

To set up the control problem, let furthermore $g\colon(\yn,\infty)\to\R $ be a measurable function with $g(\yn+)<0$, which we assume to be continuous in order to avoid some technicalities. 
To avoid trivialities, we furthermore assume that 
\[y_1\coloneqq\inf\{y>\yn:g(y)>0\}<\infty.\]
For our later considerations, we need to make sure that there exists a constant level $\bet>y_1$ such that it is never optimal to wait for the forest stand to become larger than $\bet$. 
Here, $\bet$ should be known to the decision maker no matter what the underlying drift is. 
There are many ways to realize this. 
One way is to assume that for all\ $y\geq \bet$ and all $b\in\Sigma$, it holds that
\[
\frac{g(y)}{\xi_b(y)}\le \sup_{z\in(\yn,\bet]}\frac{g(z)}{\xi_b(z)}, 
\]
where $\xi_b$ is the function given in \eqref{eq:xi} below, which we assume in the following. 
A trivial sufficient condition is 
\begin{equation*}
g(\bet)>0\quad \mbox{ and $\quad g(y)\leq g(\bet)$ for all $y>\bet$}
\end{equation*}
if -- for example -- the timber buyer has no interest in buying more than $\bet$ units at once. 
For an increasing sequence of stopping times $K=(\tau_n)_{n\in\N}$ with $\lim_{n\to\infty}\tau_n=\infty$, we consider the controlled process $X^K$ with dynamics \eqref{eq:dynamcs} for $\tau_n\leq t<\tau_{n+1}$ and 
\[X_{\tau_n}^K=X_{\tau_n-}^K+(\yn-X_{\tau_n-}^K)=\yn.\]
In other words, $X^K$ behaves like a usual diffusion process between the stopping times and is in the stopping times restarted in $\yn$.
We call such a strategy $K=(\tau_n)_{n\in\N}$ \emph{admissible impulse control strategy} if $X_{\tau_n-}^K\geq \yn$ for all\ $n\in\N$. 
We denote the set of all such strategies by $\mathcal K$. 
The problem we study is to maximize over all $K\in\mathcal K$ the long-term average value
\begin{equation}\label{eq:problem}
\liminf_{T\to\infty}\frac{1}{T} \E\left[\sum_{n:\tau_n\leq T}g(X_{\tau_n-}^K)\right].
\end{equation}
We denote the optimal value, i.e., the supremum over all the previous values, by 
\[\Phi=\Phi(b).\]
Note that also impulses of size 0 are admissible when $X_{\tau_n-}^K=\yn$. 
For the ease of exposition, we use the convention that this corresponds to no intervention, so that such an impulse gives net reward $g(\yn):=0$ (making $g$ discontinuous in $\yn$ as $g(\yn+)<0$).



\section{Solution for known $b$}\label{sec:known_b}
In case the drift $b\in \Sigma$ is known and may be used in the construction of optimal strategies, the problem is an ordinary impulse control problem and can be solved using standard approaches. 
However, note that we are interested in the case of quite general drift $b\in\Sigma$, which is -- in the ergodic case -- not treated that often, see the literature review at the end of this section. 
Let us note that generalization of the results given below were recently obtained in \cite{christensen2021competition} for a setting with additional running costs using more advanced ergodic techniques. 
For the sake of completeness, we give an elementary proof here.

For this section, we assume the coefficients of $X$ to be known, and therefore so are the speed measure $M$ and scale function $S$, which are given by their densities 
(see e.g. \cite{borodin-salminen}) 
\begin{align*}
	m(x)=\frac{1}{\sigma^2(x)}\exp\left({\int^{x}_0\frac{2b(y)}{\sigma^2(y)}dy} \right),\;\;s(x)=\exp \left({-\int^{x}_0\frac{2b(y)}{\sigma^2(y)}dy}\right).
\end{align*}
We write
\begin{equation}\label{eq:xi}
\xi(x)\coloneqq 2\int_{\yn}^xM((-\infty,y])\d S(y)\ =\ 2\int_{\yn}^x\frac{1}{\sigma^2(y)\rho(y)}\int_{-\infty}^y\rho(z)\d z \d y.
\end{equation}
The second equality immediately holds by standard results for diffusion processes by noting that the invariant density $\rho$ is -- up to normalization -- the density of\ $M$, see \cite{borodin-salminen}.
We remark that $\xi$ depends on the unknown drift $b$ via the invariant density, so that we also write $\xi=\xi_b$ if necessary. 
The following result gives a probabilistic interpretation of $\xi$:


\begin{lemma}\label{lem:martingale_threshold}
	Let\ $y>\yn$. For all\ $b\in\Sigma$ and 
	all stopping times $\tau$ with
	$\tau\leq \tau_y\coloneqq\inf\{t:X_t\geq y \}$, 
	for the uncontrolled process $X$ it holds that
	\[\E\left[\xi(X_\tau)\mid X_0=\yn\right]=\E\left[\tau\mid X_0=\yn\right].\]
	In particular, for
	all $y>\yn$ it holds 
	\[\xi(y)=\E\left[\tau_y\mid X_0=\yn\right]. \]
\end{lemma}
\begin{proof} The claim is well-known, see, e.g., \cite[Proposition 2.6]{helmes2017continuous}.
\end{proof}

The heuristic solution to the problem is now not difficult: we let the forest stand grow until it reaches a level $y$ and then cut it down, giving the (net) reward $g(y)$. 
The reward per expected time unit is therefore 
\[\frac{g(y)}{\xi(y)}=\frac{g(y)}{\xi_b(y)},\]
so that a natural candidate for an optimal level is each $y$ maximizing this value. 
The main challenge to verify this for these kind of problems is the transversality condition. 
For this we use the following lemma which is based on a coupling argument. 

\begin{lemma}\label{lem:reflected}
	 Let $X^{\yn}$ denote the stochastic process with
	\begin{itemize}
		\item diffusion dynamics \eqref{eq:dynamcs} on $(-\infty,\yn)$,
		\item when $X^{\yn}$ is in a state $\geq0$, it is restarted in state $-1$ by an immediate impulse.
	\end{itemize}
	Then, for each admissible impulse control strategy $K$ and each $T>0$, it holds that
	\[X^{\yn}_T\leq_{st} X^K_T,\]
	where $\leq_{st}$ denotes stochastic ordering. 
\end{lemma}

%

With this, we may prove the following general result. 

\begin{proposition}\label{prop:standard_control}
Assume that $b\in\Sigma$. 
Then, the value $\Phi=\Phi(b)$ for the problem \eqref{eq:problem} is given by 
	\begin{equation}\label{sol:standard_control}
	\Phi(b)=\sup_{y\in[y_1,\bet]}\frac{g(y)}{\xi_b(y)}.
	\end{equation}
Furthermore, whenever $y^*=y^*_b$ denotes a maximizer of the function  $[y_1,\bet]\to\R,\, y\mapsto \frac{g(y)}{\xi_b(y)}$, the following sequence $\hat K=(\hat\tau_n)_{n\in\N}$ is optimal:
	\[\hat\tau_n=\inf\{t\geq \hat\tau_{n-1}:X_t\geq y^* \},\;\hat\tau_0\coloneqq0. \] 
\end{proposition}

As mentioned above, the previous result is not surprising and related results do exist in the literature. For example, a Faustmann-model with discounting is presented in \cite{alvarez2004class}, where sufficient conditions for the existence of unique optimal thresholds are worked out in terms of the fundamental solutions of the underlying diffusions. In a general linear diffusion framework, ergodic impulse control problems are discussed in \cite{helmes2017continuous} and \cite{helmes2018weak} and sufficient conditions  for the optimality of threshold strategies are given, unifying many more specific results in the literature. In contrast to the model considered here, in their framework the decision maker can freely choose the target state at the control times as well. An additional slight technical difference is that minimization problems (instead of maximization problems as discussed here) are studied.

\begin{remark}
{\color{cs}As is to be expected} for ergodic problems, the value does not depend on the initial distribution of the diffusion. 
Furthermore, the renewal character of the situation is reflected in the stationarity of the optimal strategy given above. 
However, note that for the long term average criterion \eqref{eq:problem}, optimal strategies are highly non-unique: we may use quite arbitrary strategies at the beginning; just in the long run we have to switch to (asymptotically) optimal strategies. 
This {\color{cs}basic} observation will be used for the data-driven strategies constructed in Section \ref{sec:data_driven} below. 
\end{remark}

\section{Nonparametric estimation of the optimal value}\label{sec:rhobest}
The result in Proposition \ref{prop:standard_control} solves problem \eqref{eq:problem} for a \emph{known} drift function $b$. 
As we are, however, interested in finding a method for obtaining a strategy without this knowledge, we have to optimize the function using estimated values. 
This auxiliary statistical problem is treated in this section. 
Throughout this section, we assume that we are given a continuous record of observations $(X_s)_{0\leq s\leq T}$, $T>0$, of the uncontrolled diffusion process as introduced in Definition \ref{def:B}, where the drift $b\in{\color{cs}\Sigma(\C,A,\gamma,\sigma)}$ is unknown. 
Note that the diffusion coefficient $\sigma^2$ is identifiable using the quadratic variation of the semi-martingale $X$.
To simplify the exposition, we suppose that $\sigma^2(\cdot)$ is a known function as given above and that $X$ is started in the stationary distribution.

Let us introduce the operator
\[\Phi\colon {\color{cs}\Sigma(\C,A,\gamma,\sigma)}\to [0,\infty),\qquad b\mapsto \sup_{y\in[y_1,\bet]}\frac{g(y)}{\xi_b(y)},\]
with corresponding maximizers $y^*=y^*_b$, for $\xi=\xi_b$ given in \eqref{eq:xi}. 
Proposition \ref{prop:standard_control} states that $\Phi(b)$ is the optimal reward per time unit obtainable in problem \eqref{prop:standard_control} with known drift $b$, and $y^*_b$ is the optimal threshold.
In the current situation of unknown drift $b$, we aim at finding an estimator $\y$ for the threshold such that the expected regret per time unit
\[\Phi(b)-\frac{g}{\xi_b}(\y)\ =\ \left|\Phi(b)-\frac{g}{\xi_b}(\y)\right|\ \mbox{ is becoming `small' in $T$}.\]
The result \eqref{sol:standard_control} suggests that the crucial point for constructing optimal strategies in a purely data-driven way consists in replacing the function $\xi_b$ defined according to \eqref{eq:xi} by a sample-based analogue. 
We start with stating one first auxiliary result on the function $\xi=\xi_b$. 
The proof can again be found in the appendix. 

\begin{lemma}\label{lem:estimate_xi}
	There exist positive constants $M_1,M_2$ such that
	\[M_1\ \leq\ \inf_{b\in{\color{cs}\Sigma(\C,A,\gamma,\sigma)}}\inf_{x\in[y_1,\bet]}\xi_b(x)\ \leq\ \sup_{b\in{\color{cs}\Sigma(\C,A,\gamma,\sigma)}}\sup_{x\in[y_1,\bet]}\xi_b(x)\ \leq\ M_2.
	\]
\end{lemma}

We proceed by constructing estimators $\hat\xi_T$ of the functional $\xi_b$ which are based on estimators of the invariant density $\rho_b$.
One first candidate estimator relies on diffusion local time $L_T^\bullet$ {(see, e.g., Chapter 9 in \cite{legall16})}
which may be introduced via the following approximation result, holding a.s.~for every $a\in\R$ and $T\geq0$,
\[
L_T^a(X) = \lim_{\ep\to0}\frac{1}{\ep}\int_0^T\mathds{1}\left\{a\leq X_s \leq a+\ep\right\}\d\langle X\rangle_s,\quad a\in\R.
\]
The above representation in particular suggests the interpretation of diffusion local time as the derivative of an empirical distribution function.
{Note further that $\E L_T^a(X)=T\rho_b(a)\sigma^2(a)$, $a\in\R$.
Both observations motivate the use of} $L_T^\bullet$ for constructing an (unbiased) estimator $\rho_T^\circ$ of $\rho_b$ by letting
\[
\rho_T^\circ(a)\coloneqq\frac{L_T^a(X)}{T\sigma^2(a)},\quad a\in\R.
\]
While the estimator $\rho_T^\circ$ allows for the seamless application of probabilistic results on diffusion local time, it is of minor interest with regard to real-world applications.
In particular, it is not obvious how the local time estimator could be replaced in situations where only discrete observations of the diffusion process $X$ are available. 
One very useful step towards finding practically more feasible estimators of $\rho_b$ (and, as a result, the functional $\xi_b$) consists in replacing $\rho_T^\circ$ with a kernel-type invariant density estimator.
In order to analyse the approximation properties of the kernel estimator, we however need to impose some regularity on $b$ and $\rho_b$. 
For doing so, we consider H\"older classes defined as follows:


\begin{defn}\label{def:Holder}
	Given $\beta,\mathcal{L}>0$, denote by $\mathcal{H}_{\R}(\beta,\mathcal{L})$ the \emph{H\"older class} (on $\R$) as the set of all functions $f\colon\R\to\R$ which are $l:=\lfloor \beta \rfloor$-times differentiable and for which
	\begin{align*}
		\|f^{(k)}\|_{\infty}&\le\mathcal{L}  \qquad\qquad\forall\, k=0,1,\ldots,l,\\
		\|f^{(l)}(\cdot+t)-f^{(l)}(\cdot)\|_{\infty}&\le\mathcal{L}|t|^{\beta- l}  \qquad \forall\, t\in\R.
	\end{align*}
	Set $\Sigma(\beta,\mathcal{L})\coloneqq\left\{b\in{\color{cs}\Sigma(\C,A,\gamma,\sigma)}\colon\ \rho_b\in\mathcal{H}_\R(\beta,\mathcal{L})\right\}$.
\end{defn}

Note that the assumption that $b\in{\color{cs}\Sigma(\C,A,\gamma,\sigma)}$ again implies certain regularity properties on the invariant density in the sense of Definition \ref{def:Holder}. 
More precisely, if $b\in{\color{cs}\Sigma(\C,A,\gamma,\sigma)}$, the invariant density $\rho_b$ is bounded and Lipschitz continuous due to \eqref{eq:reg_rho_b} which in turn means $b\in \Sigma(1,\ML)$.
Considering the class of drift coefficients $\Sigma(\beta,\mathcal{L})$, we use kernel functions fulfilling the following conditions,
\begin{equation}\label{kernel}
\begin{array}{r@{}l}
&{}\bullet\quad Q\colon\R\rightarrow \R \text{ is Lipschitz continuous and symmetric},\\ [3pt]
&{}\bullet\quad\operatorname{supp}(Q)\subset [-1/2,1/2],\\[3pt]
&{}\bullet\quad Q\text{ is of order }\lfloor\beta \rfloor\text{, i.e., } \int Q(u)\d u=1 \text{ and }\\
&{}\hspace*{8em} \int Q(u)u^j\d u=0 \text{ for } j=1,\ldots,\lfloor\beta\rfloor.
\end{array}
\end{equation}

\begin{lemma}[concentration of invariant density estimators]\label{lem:csi}
	Let $\X$ be a diffusion as in Definition \ref{def:B} {and $\com\subset \R$ be compact}.
	\begin{enumerate}
		\item[$\operatorname{(i)}$]
		There is a positive constant $\kappa$ such that, for any $T\geq 1$,
		\[
		\sup_{b\in{\color{cs}\Sigma(\C,A,\gamma,\sigma)}}
\E_b\left[\left\|\rho_T^\circ-\rho_b\right\|_{L^1(\com)}\right]\le \frac{\kappa}{\sqrt T}.
		\]
		\item[$\operatorname{(ii)}$]
		Assume that $b\in\Sigma=\Sigma(\beta,\mathcal L)$, for some $\mathcal L>0,\,\beta\geq 1$, and let $Q$ be a kernel function fulfilling \eqref{kernel}.
		Given some positive bandwidth $h$, define the estimator 
		\[
		\rho_{T,Q}(h)(x)\coloneqq \frac{1}{Th}\int_0^TQ\left(\frac{x-X_u}{h}\right)\d u,\quad x\in\R.
		\] 		
		Then, there is a positive constant $\nu$ such that, for any $T>0$,
		\[ 
		\sup_{b\in\Sigma(\beta,\mathcal L)}\E_b\left[\left\|\rho_{T,Q}(h) - \rho_b\right\|_{L^1(\com)}\right]
		\le \frac{\nu}{\sqrt T}+\frac{\mathcal Lh^\beta}{\lfloor \beta\rfloor!}\int |u^{\beta }Q(u)|\d u.
		\]
	\end{enumerate}
\end{lemma}
The arguments of the proof are similar to the investigation in Section 5 in \cite{cacs18} where density estimation in supremum-norm is considered.
For the $L^1$ risk considered here, we obtain a faster convergence rate, without any logarithmic terms.
To shorten notation, let 
\[\hat\rho_T(z)\coloneqq \rho_{T,Q}(T^{-1/2})(z)=\frac{1}{\sqrt T}\int_0^TQ\left(\sqrt T(z-X_s)\right)\d s,\quad z\in\R.\]
By just plugging in and using the bound from Lemma \ref{lem:estimate_xi}, we now define
\begin{equation*}
\begin{split}
\xi_T^\circ(x)&\coloneqq \max\left\{M_1,\ 2\int_{0}^x\frac{1}{(\rho_T^\circ(y)\vee a)\sigma^2(y)}\int_{0}^y\rho_T^\circ(z)\d z\d y\right\},\\
\hat\xi_T(x)&\coloneqq \max\left\{M_1,\ 2\int_{0}^x\frac{1}{(\hat\rho_T(y)\vee a)\sigma^2(y)}\int_{0}^y\hat\rho_T(z)\d z\d y\right\},
\end{split}\quad x\in\R,
\end{equation*}
for $a=\inf_{b\in{\color{cs}\Sigma(\C,A,\gamma,\sigma)}}\min_{x\in [y_1,\bet]}\rho_b(x)>0$. 

{\color{cs}
\begin{remark}
Note that, taking into account Definition \ref{def:B}, for any $b\in\Sigma=\Sigma(\C,A,\gamma,\sigma)$ and $x\in[0,\bet]$, $C_{b,\sigma}\sigma^2(x)\rho_b(x)=\exp\left(2\int_0^x\frac{b(v)}{\sigma^2(v)}\d v\right)$ is bounded from below by
\begin{align*}
\exp\left(-2\int_0^x\frac{|b(v)|}{\sigma^2(v)}\d v\right)&\ge \exp\left(-2\underline\nu^{-2}\C\int_0^x(1+v)\d v\right)\\
&\ge\exp\left(-2\underline\nu^{-2}\C\int_0^\bet (1+v)\d v\right),
\end{align*}
implying that, for any $b\in\Sigma$,
\[
\min_{x\in[y_1,\bet]}\rho_b(x)\ge\min_{x\in[0,\bet]}\rho_b(x)\ge C_{b,\sigma}^{-1}\underline\nu^{-2}\exp\left(-\underline\nu^{-2}\C\left(2\bet+\bet^2\right)\right).
\]
Since tedious but straightforward calculations show that
\[
\sup_{b\in\Sigma(\C,A,\gamma,\sigma)}C_{b,\sigma}\le\underline\nu^{-2}\exp\left(\underline\nu^{-2}2A\C(1+A)\right)\left(2A+\gamma^{-1}\right),\]
the previous considerations allow for a concrete specification of $a$. 
Note however that we just introduce this parameter to guarantee that the estimators of $\rho$ do not attain unreasonably low values. 
Choosing a constant $a$ that is too small would only influence the following results in the sense that the specified constants would have to be increased. 
In addition, for a sufficiently large amount of data, the exact choice of the parameter $a$ will often not be relevant as we will usually automatically see $\hat\rho_T(y)\geq  a$.
\end{remark}}

We now take
\[\yy \in\arg\max_{y\in[y_1,\bet]}\frac{g}{\xi_T^\circ}(y)\quad\text{ and }\quad \y\in\arg\max_{y\in[y_1,\bet]}\frac{g}{\hat\xi_T}(y). \]
Using this, we obtain:

\begin{proposition}\label{prop:nonparam_estimator}
	There exist constants $C_1,C_2>0$ such that
	\begin{align*}
	\sup_{b\in{\color{cs}\Sigma(\C,A,\gamma,\sigma)}}\E_b\left[\Phi(b)-\frac{g}{\xi_b}(\yy)\right]	&\le\frac{C_1}{\sqrt T},\\ 
	\sup_{b\in\Sigma(\beta,\LL)}\E_b\left[\Phi(b)-\frac{g}{\xi_b}(\y)\right]&\le\ \frac{C_2}{\sqrt T}. 
	\end{align*}
\end{proposition}
\begin{proof}
	Note that
	\begin{align*}
	\Phi(b)-\frac{g}{\xi_b}(\y)
	&=\frac{g}{\xi_b}(y^\ast_b)-\frac{g}{\hat\xi_T}(\y)+\frac{g}{\hat\xi_T}(\y)-\frac{g}{\xi_b}(\y)\\
	&=\max_{y\in[y_1,\bet]}\frac{g}{\xi_b}(y)-\max_{y\in[y_1,\bet]}\frac{g}{\hat\xi_T}(y)+\frac{g}{\hat\xi_T}(\y)-\frac{g}{\xi_b}(\y)\\
	&\le 2\max_{y\in[y_1,\bet]}\left|\frac{g}{\xi_b}(y)-\frac{g}{\hat\xi_T}(y)\right|.
	\end{align*}
	We are thus led to analysing
	\begin{align*}
	\E_b\left[\max_{y\in[y_1,\bet]}\Big|\frac{g}{\hat\xi_T}(y)-\frac{g}{\xi_b}(y)\Big|\right]
	&\le\E_b\left[\max_{y\in[y_1,\bet]}g(y)\Big|\frac{1}{\hat\xi_T}(y)-\frac{1}{\xi_b}(y)\Big|\right]\\
	&\le  LM_1^{-2}\E_b\left[\max_{y\in[y_1,\bet]}\left|\hat\xi_T(y)-\xi_b(y)\right|\right],
	\end{align*}
	where $L$ denotes the maximum of $g$ on $[y_1,\bet]$.
To shorten notation, let $\hat I_T(y)\coloneqq \int_0^y\hat\rho_T(z)\d z$, $I(y)\coloneqq\int_0^y\rho_b(z)\d z$.
Then, for any $x\in[y_1,\bet]$ and $\com\coloneqq[0,\bet]$,
\begin{align*}
\left|\hat\xi_T(x)-\xi_b(x)\right|
&\le 2\left|\int_{0}^x\left(\frac{\hat I_T(y)}{\left(\hat\rho_T(y)\vee a\right)\sigma^2(y)}-\frac{I(y)}{\rho_b(y)\sigma^2(y)}\right)\d y\right|\\
	&\le 2\left|\int_{0}^x\frac{\hat I_T(y)-I(y)}{\left(\hat\rho_T(y)\vee a\right)\sigma^2(y)}\d y\right|\\
	&\hspace*{3em}
	+2\left|\int_{0}^x\left(\frac{I(y)}{\left(\hat\rho_T(y)\vee a\right)\sigma^2(y)}-\frac{I(y)}{\rho_b(y)\sigma^2(y)}\right)\d y\right|\\
	&\le \int_{\com}\left|\hat\rho_T(z)-\rho_b(z)\right|\d z\cdot \left(2\int_{0}^x\frac{\d y}{\left(\hat\rho_T(y)\vee a\right)\sigma^2(y)}\right)\\&\hspace*{3em}
	+2\left|\int_{0}^xI(y)\frac{\rho_b(y)-\left(\hat\rho_T(y)\vee a\right)}{\rho_b(y)\left(\hat\rho_T(y)\vee a\right)\sigma^2(y)}\d y\right|\\
	&\le 
2\|\hat\rho_T-\rho_b\|_{L^1(\com)}\left(\frac{\bet}{a\underline\nu^2}\left(1+\frac{1}{a}\right)\right).
	\end{align*}
The claim now immediately follows from Lemma \ref{lem:csi}.
\end{proof}

The previous result can be interpreted in terms of a statistical approach to the Faustmann problem as follows: 
if we first collect data for the uncultivated forest stand for $T$ time units and then start using the estimator for the unknown optimal threshold based on these observations, then the expected regret per time unit is of order $T^{-1/2}$. 
A more realistic problem is of course how to simultaneously learn and exploit. 
This is treated in the next section.

\section{Data driven impulse controls}\label{sec:data_driven}
When we try to use the estimator for $\rho$ to approximate the optimal strategy as described in the previous section, we face a classical trade-off between exploration and exploitation. 
On the one hand, we want to maximize our expected reward and therefore use one-sided impulse control strategies with estimated optimal threshold $\widehat{y}_T$ as introduced in the previous section. 
On the other hand, using this strategy, we cannot learn about {the diffusion dynamics} on the interval between this threshold and $\bet$ and therefore our procedure cannot even be expected to converge. 

\begin{figure}
	\centering
	\includegraphics[width=0.9\linewidth]{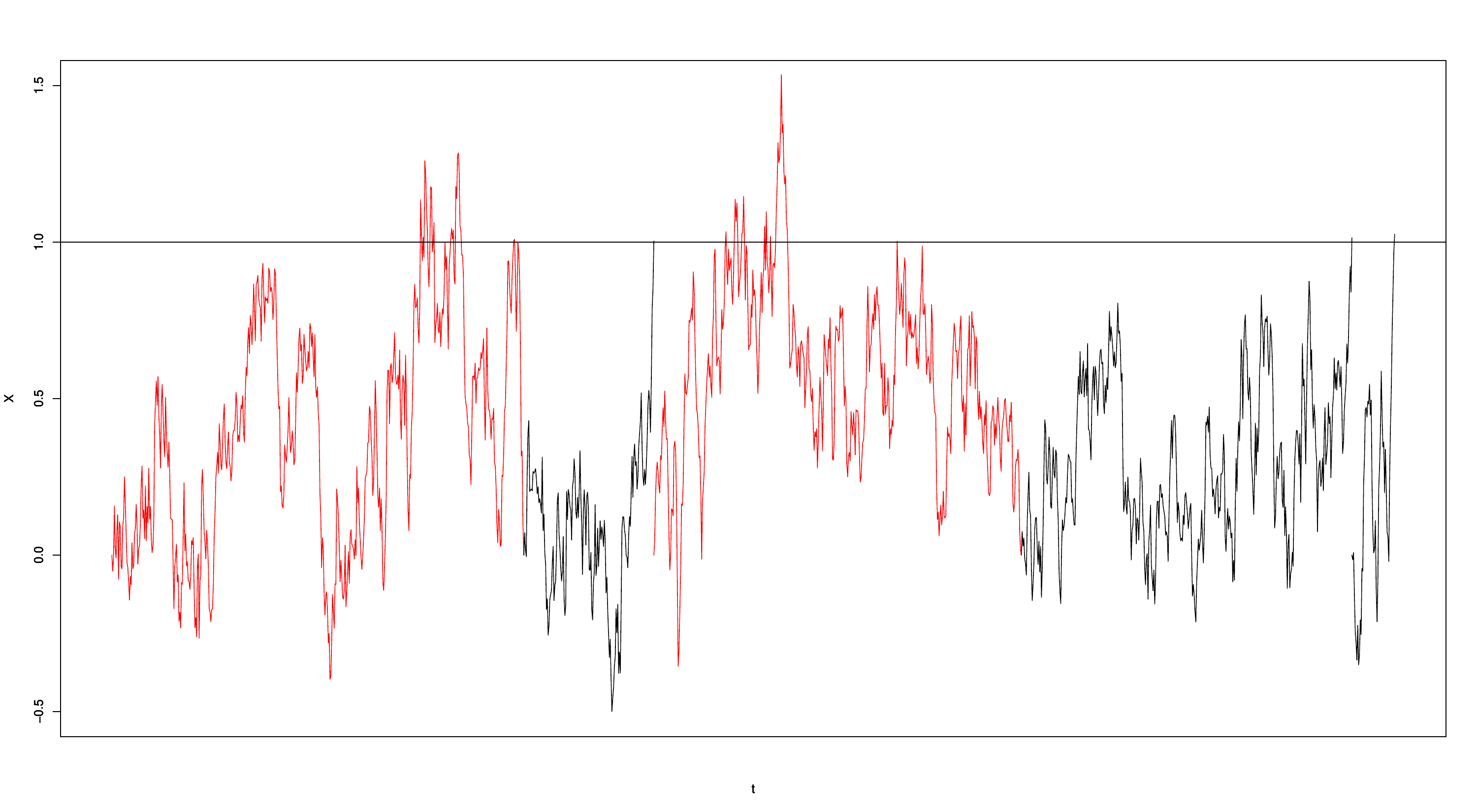}
	\caption{A path controlled using a data driven impulse strategy with exploration (red) and exploitation (black) periods}
	\label{fig:plot_expl_exploit}
\end{figure}

We propose a strategy with an decaying $\epsilon$-greedy-inspired action selection as follows, see Figure \ref{fig:plot_expl_exploit}: 
we use stopping times $\tau_1,\tau_2,\ldots$ to separate the time into different periods. 
Each of these periods is either an \emph{exploration period} or an \emph{exploitation period}. 
In each exploration period, we let the process run until it first reaches the upper boundary $\bet$ and then visits $\yn$ again. 
Here, many other choices would be possible. 
It would, for example, be natural to control the process back to $\yn$ whenever it reaches $\bet$. 
We choose our definition primarily for mathematical convenience, see below. 

We denote by $S_t$ the (random) time spent in the exploration periods until $t$, i.e., the process $(S_t)_t$ has linear growth in the exploration periods and is constant in the exploitation periods. 
We denote the left inverse of $t\mapsto S_t$ by $S^{-1}$ and write ${X}':=(X^K_{S^{-1}_s})_{s\geq 0}$. 
As exploitation periods start and end in\ $\yn$, the process $X'$  is -- by the strong Markov property -- a diffusion process with same dynamics as $X$ (not killed as long as $S_t\to \infty$). 
Based on this observation, we may use one of the estimators $\widehat{y}'$ from\ Section \ref{sec:rhobest}, but now for $X'$ instead of $X$, and write \[\widehat{y}_t:=\widehat{y}'_{S_t\wedge mt^{2/3}}\]
for the estimated optimal threshold at time $t$. Here, for technical reasons only (see the proof below), at time $t$, we limit the data from the exploration periods to an observation length of $s=mt^{2/3}$ for some constant $m$ specified below.

Note that the estimated threshold is constant in the exploitation periods, i.e. in exploitation periods $[\tau_{n-1},\tau_{n})$ it holds that
\[\widehat{y}_t=\widehat{y}_{\tau_{n-1}}. \]
There, we use a threshold impulse control strategy with threshold $\widehat{y}_T$ and then again restart the process in $\yn$.  We call such a strategy $K$ combining {exploration and exploitation periods} a \emph{data driven impulse strategy}.

The main question now is how to balance the time $S_T$ spent in exploration periods and $T-S_T$ in exploitation periods. 
Obviously, $S_T$ needs to converge to $\infty$ and $S_T/T$ to 0 to obtain convergence to the optimal value. 
We now analyse the rate by balancing the error rates from exploration and exploitation:\ 
In the exploration periods, we receive no payoff, hence the cumulated regret until $T$\ is of order $S_T$. 
In the exploitation periods, the regret is caused by choosing a suboptimal threshold. 
By Proposition \ref{prop:nonparam_estimator}, the cumulated regret is bounded by $c\int_0^T S_t^{-1/2}dt$. 
Balancing out leads to choosing $S_T\approx T^{2/3}$. 
More precisely, we use the following condition \eqref{eq:S_t-bed}. The existence of such strategies is straightforward using elementary renewal theory, see \cite[Appendix B]{christensen2021learning} for a very similar construction.

%

\begin{thm}\label{thm:main}
	Let $K$ be a data driven impulse strategy as described above with exploration time $S_T$ that grows as $T^{2/3}$ in the following sense: there exist constants $m,M>0$ such that 
	\begin{align}\label{eq:S_t-bed}
		\mathds P_b(S_T\leq mT^{2/3})\lesssim T^{-1/3}\mbox{ and }\limsup_{T\to\infty}T^{-2/3}\E_bS_T< M
	\end{align}
	for all\ $b\in\Sigma(\beta,\LL)$.
	Then, the difference of optimal reward rate $\Phi(b)$ and the expected data driven reward rate is of uniform order $O\big({T^{-1/3}}\big)$, that is, there exists $D>0$ such that 
	\[\limsup_{T\to\infty}{T^{1/3}}\left(\Phi(b)-\frac{1}{T} \E_b \sum_{n:\tau_n\leq T}g(X_{\tau_n-}^K)\right)\ \leq\ D.\]
	for all $b\in\Sigma(\beta,\LL)$.
\end{thm}


\begin{proof}
	We fix some drift $b\in\ \Sigma$.
	We write $Z_t=t-S_t$ for the time in the exploitation periods until\ $t$. 
	Keeping the convention $g(\yn)=0$ in mind, it holds that
	\begin{align*}
	&\E_b \sum_{n:\tau_n\leq T}g(X_{\tau_n-}^K)\ =\ \E_b \sum_{\substack{n:\tau_n\leq T\\ \text{exploitation period}}}g(\widehat{y}_{\tau_{n-1}})\\
	&\hspace*{1.5em}=\E_b \sum_{\substack{n:\tau_n\leq T\\ \text{exploitation period}}}\frac{g(\widehat{y}_{\tau_{n-1}})}{\E_b\left[\tau_n-\tau_{n-1}\mid\F_{\tau_{n-1}}\right]}{\E_b\left[\tau_n-\tau_{n-1}\mid\F_{\tau_{n-1}}\right]}\\
	&\hspace*{1.5em}=\E_b \sum_{\substack{n:\tau_n\leq T\\ \text{exploitation period}}}\frac{g(\widehat{y}_{\tau_{n-1}})}{\xi_b(\widehat{y}_{\tau_{n-1}})}{\E_b\left[\tau_n-\tau_{n-1}\mid\F_{\tau_{n-1}}\right]}\\
	&\hspace*{1.5em}=\E_b \sum_{\substack{n:\tau_n\leq T\\ \text{exploitation period}}}\E_b\left[\frac{g(\widehat{y}_{\tau_{n-1}})}{\xi_b(\widehat{y}_{\tau_{n-1}})}{(\tau_n-\tau_{n-1})}\mid\F_{\tau_{n-1}}\right]\\
	&\hspace*{1.5em}=\E_b \sum_{\substack{n:\tau_n\leq T}}\E_b\left[\int_{\tau_{n-1}}^{\tau_n}\frac{g}{\xi_b}(\widehat{y}_t)\d Z_t\mid\F_{\tau_{n-1}}\right]\\
	&\hspace*{1.5em}\geq \E_b \sum_{\substack{n:\tau_{n-1}\leq T}}\E_b\left[\int_{\tau_{n-1}}^{\tau_n}\frac{g}{\xi_b}(\widehat{y}_t)\d Z_t\mid\F_{\tau_{n-1}}\right]-c\ = \E_b\left[\int_{0}^{\gamma_T}\frac{g}{\xi_b}(\widehat{y}_t)\d Z_t\right]-c,
	\end{align*}
	where $\gamma_T=\sup\left\{\tau_{m+1}: \tau_m\leq T\right\}\geq T$ and $c=\Phi(b)\xi_b(\bet)$. 
Therefore,
	\[\E_b \sum_{n:\tau_n\leq T}g(X_{\tau_n-}^K)\ \geq\  \E_b \left[\int_{0}^{T}\frac{g}{\xi_b}(\widehat{y}_t)\d Z_t\right]-c.\]
	We obtain
	\begin{align*}
	\Phi(b)T-\E_b \sum_{n:\tau_n\leq T}g(X_{\tau_n-}^K)&\leq\ \Phi(b)T-\E_b\left[\int_{0}^{T}\frac{g}{\xi_b}(\widehat{y}_t)\d Z_t\right]+c\\
	&=\ \Phi(b)\E_b S_T+\E_b\left[\int_{0}^{T}(\Phi(b)-\frac{g}{\xi_b}(\widehat{y}_t))\d Z_t\right]+c\\
	&\leq \Phi(b)\E_b S_T+\E_b\left[\int_{0}^{T}(\Phi(b)-\frac{g}{\xi_b}(\widehat{y}_t))\d t\right]+c\\
	&= \Phi(b)\E_b S_T+\int_{0}^{T}\E_b\left[\Phi(b)-\frac{g}{\xi_b}(\widehat{y}_{S_t}')\right]\d t+c.
	\end{align*}
	Proposition \ref{prop:nonparam_estimator} yields that
\begin{align*}
	\E_b\left[\Phi(b)-\frac{g}{\xi_b}(\widehat{y}_{S_t}')\right]\lesssim& \mathds P_b(S_t\leq mt^{2/3})+(mt^{2/3})^{-1/2}\lesssim t^{-1/3}.
\end{align*}
Putting pieces together
	\begin{align*}
	\Phi(b)T-\E_b \sum_{n:\tau_n\leq T}g(X_{\tau_n-}^K)&\lesssim\ T^{2/3}+\int_{0}^{T}t^{-1/3}\d t+c\\
	&\lesssim T^{2/3}.
	\end{align*}
	Dividing by $T$ yields the result. 	
\end{proof}

\section{Conclusion}\label{sec:concl}
We studied a statistical version of the stochastic Faustmann timber harvesting problem. 
Our approach relied on methods both from stochastic impulse control and nonparametric statistics of diffusion processes. 
Based on this, we proposed a method to deal with the exploration vs.~exploitation problem and found a rate of convergence of $T^{^{-1/3}}$. 
The proposed method separates exploration and exploitation periods.
{An important feature of the studied control problem is the fact that the solution for known dynamics of the underlying diffusion process depends on the unknown drift via the invariant density. 
In the framework of continuous observations, this quantity can be estimated with a parametric rate of convergence.
As a consequence, we do not have to pay for allowing for a very general nonparametric structure (as compared to standard parametric regimes) in the proposed data-driven control strategy.}
{
Moreover, since we were able to reduce our analysis to bounding the $L^1$ risk of the density estimator, no logarithmic factors are contained in the convergence rate (in contrast to controlling the $\sup$-norm risk). 
In case that knowledge of the drift itself is required for proposing the optimal control strategy, more restrictive assumptions would be necessary and the convergence rate for the estimation procedure would decrease. }

Of course, many refinements of this data-driven method could be studied. 
The main reason for considering our approach was technical: the process $X'$ used for the estimation is a diffusion process itself. 
Furthermore, many alternative models could be considered.\ 
Keeping the motivation from natural resource management in mind, another possibility could be to leave a (small, representative) part of the forest stand uncultivated to learn the dynamics of the forest growth dynamics from this and apply the observations directly to the cultivated part. 
This makes the situation easier as the exploration and exploitation problems are naturally separated.

It is of some interest to describe the problem from a bandit-perspective. 
There, the choice of the intervention level $y$ in each cycle could be seen as the arm chosen by the decision maker producing random reward $g(y)/\tau_y$. 
Continuous ($\mathcal X$-armed) bandit problems of a similar type have been studied extensively in the literature over the last years, see \cite{auer2007improved,bubeck2011x,locatelli2018adaptivity} to name just a few. 
There are, however, some differences. 
First, we do not assume smoothness of the reward ($g$ was assumed to be continuous just for simplicity). 
Second, our reward structure is of a nonstandard form and is unbounded in general. Third, the error term depends on the chosen arm and the underlying distributions.

Overall, this paper makes a first (small) step to bringing together techniques from stochastic control with methods from statistics for stochastic processes to find a way to both learn the dynamics of the underlying process and control in a reasonable way at the same time. 
Our hope is that it can stimulate future research in this field. 


\begin{appendix}

\section{Some Proofs}
\begin{proof}[Proof of Lemma \ref{lem:reflected}] Let $K$ be an admissible impulse control strategy. As $X^{\yn}_T\ \leq_{st}\ X^K_T$ is just a distributional property, we are free to construct a realization of $\tilde X^{\yn}=\tilde X^{\yn,K}$ based on the process $X^K$ using the following coupling: 
	Let $\tilde X^{\yn}$ and $X^K$ run coupled until a state $\geq \yn$ is reached. 
	Then, we restart $\tilde X^{\yn}$ in\ $-1$ and let both processes run independently until the first time the two paths meet again. 
	Then, we couple the paths again and follow this rule. Then $X^0$ and $\tilde X^{\yn}$ have the same law. Furthermore, for each $t$ and each\ $\omega$, we have $\tilde X^{\yn}_t(\omega)\leq X^K_t(\omega)$, proving in particular, $\tilde X^{\yn}_T\ \leq_{st}\ X^K_T$ and therefore $ X^{\yn}_T\ \leq_{st}\ X^K_T$.
\end{proof}

\begin{proof}[Proof of Proposition \ref{prop:standard_control}]
	We first consider the process $M=(\xi(X_t)-t)_{t\geq0}$ for the uncontrolled process $X$. By \cite[Proposition 2.6]{helmes2017continuous} we see that it is a local martingale  with localizing family $(\tau_y)_{y>0}$. For a bounded stopping time $\sigma$ with $X_\sigma\geq 0$ we obtain by bounded convergence
$$
		\E_b M_\sigma=\E_b\lim_{y\to\infty}(\xi(X_{\sigma\wedge \tau_y})-\sigma\wedge \tau_y)= \lim_{y\to\infty}\E_bM_{\sigma\wedge\tau_y}=\E_bM_{0}.
$$
		Now, let\ $K=(\tau_n)_{n\in\N}$ be an admissible impulse control strategy. 
	As $X^K$ runs uncontrolled on $[\tau_{n-1},\tau_n)$ for all $n$, the strong Markov property together with the previous observation yields 
	\begin{equation*}
	\E_b\left[\xi(X_{\tau_n-\wedge T}^K)-\xi(X^K_{\tau_{n-1}\wedge T})-(\tau_n\wedge T-\tau_{n-1}\wedge T)\right] \ = \ 0,
	\end{equation*}
	where we use the convention that $\tau_0=0$.
	Rearranging terms and using dominated convergence, we obtain for all $\gamma>0$
	\begin{align*}
	0&= \ -\gamma\E_b\left[\sum_{n=1}^\infty\left(\xi(X^K_{\tau_n-\wedge T})-\xi(X^K_{\tau_{n-1}\wedge T})-(\tau_n\wedge T-\tau_{n-1}\wedge T)\right)\right]\\
	&=\ -\gamma\E_b\left[\sum_{n=1}^\infty\left(\xi(X^K_{\tau_n-\wedge T})-\xi(X^K_{\tau_{n-1}\wedge T})\right)\right]+\gamma T\\
	&=\ \E_b\left[\sum_{n:\tau_n\leq T}\gamma\left(\xi(X^K_{\tau_n})-\xi(X^K_{\tau_n-})\right)\right]+\gamma\E_b\xi(X_0^K) -\gamma\E_b\xi(X^K_{T})+\gamma T\\
	&=\ \E_b\left[\sum_{n:\tau_n\leq T}-\gamma\xi(X^K_{\tau_n-})\right]+\gamma\E_b\xi(X_0^K) -\gamma\E_b\xi(X^K_{T})+\gamma T,
	\end{align*}
	where we used $\xi(\yn)=0$ in the last step. Therefore, we obtain for all $\gamma\geq \max_{y\in[y_1,\bet]}\frac{g(y)}{\xi(y)}$ that
	\begin{align*}
	&\liminf_{T\to\infty}\frac{1}{T} \E_b\sum_{n:\tau_n\leq T}g(X_{\tau_n-}^K)\\
	&\hspace*{3em}=\ \liminf_{T\to\infty}\frac{1}{T} \Bigg(\E_b\bigg[\sum_{n:\tau_n\leq T}(g(X_{\tau_n-}^K)-\gamma\xi(X^K_{\tau_n-}))\bigg]\\
&\hspace*{15em}+\gamma\E_b\xi(X_0^K) -\gamma\E_b\xi(X^K_{T})+\gamma T\Bigg)\\
	&\hspace*{3em}\leq\ \liminf_{T\to\infty}\frac{1}{T} \left(0+\gamma\E_b\xi(X_0^K) -\gamma\E_b\xi(X^K_{T})\right)+\gamma\\
	&\hspace*{3em}=\ \liminf_{T\to\infty}\frac{1}{T} \left( -\gamma\E_b\xi(X^K_{T})\right)+\gamma\ \leq\ \gamma.
	\end{align*}
	For the last inequality, note that by Lemma \ref{lem:reflected} and the monotonicity of $\xi$
	\[\E_b\xi(X^K_{T})\ \geq\ \E_b\xi(X^{\yn}_{T})\]
	and as $b\in\Sigma$. Moreover, the auxiliary process $X^{\yn}$ is ergodic with a stationary density $\rho^0$ such that $\rho^0(x)$, up to a constant, coincides with $\rho(x)$ for $x\leq -1$ (\cite[Proposition 3.1]{helmes2017continuous}). As $\xi$ has at most linear growth on the negative part of the real axis, $\xi$ is $\rho^0$-integrable and we obtain
	\[\lim_{T\to\infty}\frac{1}{T} \E_b\xi(X^\yn_{T})\ =\ 0.\]
	This proves
	\begin{equation*}
	\Phi\ \leq \ \sup_{y\in[y_1,\bet]}\frac{g(y)}{\xi_b(y)}.
	\end{equation*}
Applying the same line of argument to $(\hat\tau_n)_{n\in\N}$, we see that the second inequality in the calculations above becomes equality. 
Noting that $X^{\hat K}$ is also ergodic, also the third inequality is indeed an equality. 
The same holds for the first one adapting the first steps in the proof. 
Putting pieces together, this proves that 
	\begin{align*}
	&\liminf_{T\to\infty}\frac{1}{T} \E_b\sum_{n:\hat\tau_n\leq T}g(X_{\hat\tau_n-}^{\hat K})\ =\ \Phi,
	\end{align*}
	i.e., $\hat K$ is optimal. 
\end{proof}

\begin{proof}[Proof of Lemma \ref{lem:estimate_xi}]
First note that by \eqref{eq:xi} the function\ $\xi_b$ is non-decreasing in $x$. 
Therefore,
	\[\inf_{x\in[y_1,\bet]}\xi_b(x)\ =\ \xi_b(y_1),\quad \sup_{x\in[y_1,\bet]}\xi_b(x)\ =\ \xi_b(\bet).\]
	Using the explicit expressions for the scale and speed function from \cite{borodin-salminen}, we furthermore obtain 
	\[\xi_b(x)\ =\ 2\int_{\yn}^{x}\int_{-\infty}^y\frac{1}{\sigma^2(z)}\exp\left(-2\int_z^y\frac{b(q)}{\sigma^2(q)}\d q\right)\d z\d y.\]
	Recall that $b(y)/\sigma^2(y)\geq\gamma$ for all $b\in\Sigma$ and $y<-A$. Furthermore, all $b\in \Sigma$ are uniformly bounded by a constant $\alpha$, i.e. $b(y)/\sigma^2(y)\geq\alpha$ for all $b\in\Sigma$ and $y\in[A,\bet]$. Therefore, it holds that
\begin{align*}
\xi_b(x)&\leq\ \xi_b(\bet)\\
&\leq\ \frac{2}{h}\int_{\yn}^{\bet}\left(\int_{-\infty}^A\e^{-2\gamma(z-y)}\d z+\int_{A}^y\e^{-2\alpha(z-y)}\d z\right) \d y\ =:\ M_2
\end{align*}
for all\ $b\in\Sigma$ and $x\in[y_1,\bet]$, proving the second bound. 
The first one can be found analogously.	
\end{proof} 

\begin{proof}[Proof of Lemma \ref{lem:csi}]
We only prove (ii). The analysis of the local time estimator and the proof of (i) are similar.
Introduce the function
\[g_y(u)\coloneqq \frac{1}{\rho_b(u)\sigma^2(u)}\int_{\R}Q\left(\frac{y-x}{h}\right)\rho_b(x)\left(\mathds{1}\{u>x\}-F_b(u)\right)\d x,\quad u\in\R,\]
$F_b$ denoting the distribution function associated with the invariant density $\rho_b$.
It\^o's formula applied to the function $\int_0^\bullet g_y(u)\d u$ and to the diffusion $X$ yields
\[
\int_{X_0}^{X_t}g_y(u)\d u=\int_0^tg_y(X_s)\d X_s+\frac{1}{2}\int_0^tg_y'(X_s)\d\langle X\rangle_s.
\]
Since $(\rho_b\sigma^2)'=2b\rho_b$ and 
\[g_y'(u)=-\frac{2b(u)}{\sigma^2(u)}g_y(u)+\frac{1}{\sigma^2(u)}\left\{Q\left(\frac{u-x}{h}\right)- \E_b\left[Q\left(\frac{y-X_0}{h}\right)\right]\right\},\ \ u\in\R,
\]
one obtains the representation
\begin{align*}
&\sqrt th\left(\rho_{t,Q}(h)(y)-\E_b\left[\rho_{t,Q}(h)(y)\right]\right)\\
&\hspace*{3em}=\frac{1}{\sqrt t}\int_0^tQ\left(\frac{y-X_s}{h}\right)\d s-\sqrt t\E_b\left[Q\left(\frac{y-X_0}{h}\right)\right]\\
&\hspace*{3em}=\ t^{-1/2}\rd_t^y+t^{-1/2}\Ma_t^y,
\end{align*}
for 
$\rd_t^y\coloneqq2\int_{X_0}^{X_t}g_y(u)\d u$ and $\Ma_t^y\coloneqq-2\int_0^tg_y(X_s)\sigma(X_s)\d W_s$.
The conditions on the class $\Sigma$ ensure that there exists a constant $M=M(\C,A,\gamma,\overline\nu,\underline\nu)$ such that, for any $b\in\Sigma(\C,A,\gamma,\sigma)$,
\[
\sup_{x\geq 0}\frac{1-F_b(x)}{\sigma^2(x)\rho_b(x)}\ \leq\ M\quad\text{ and } \quad\sup_{x\leq0}\frac{F_b(x)}{\sigma^2(x)\rho_b(x)}\ \leq\ M.
\] 
It further can be shown (see Proposition {\color{cs}7} in \cite{cacs18}) that there exists some constant $\Pi_0>0$ such that
\[
\E_b\left[\sup_{y\in\R}|\rd_t^y|\right]\ \le\ \Pi_1\sup_{y\in\mathbb Q}\lebesgue\left(\operatorname{supp}\left(Q\left(\frac{y-\cdot}{h}\right)\right)\right)\ \le\ \Pi_0h.
\]
For any $u\in\R$, it holds
\begin{align*}
\left|g_y(u)\right|^2
&=\ \left|\int_\R Q\left(\frac{y-x}{h}\right)\rho_b(x)\ \frac{\mathds{1}\{u>x\}-F_b(u)}{\rho_b(u)\sigma^2(u)}\d x\right|^2\\
&\leq\ \int_\R Q^2\left(\frac{y-x}{h}\right)\d x\int_\R \rho_b^2(x)\mathds{1}\{|y-x|\leq h\}\frac{(\mathds{1}\{u>x\}-F_b(u))^2}{\rho_b^2(u)\sigma^4(u)}\d x\\
&=\ \left\|Q\left(\frac{y-\cdot}{h}\right)\right\|_{L^2(\R)}^2\Bigg\{
\underbrace{\frac{F_b^2(u)}{\rho_b^2(u)\sigma^4(u)}\int_u^\infty\mathds{1}\left\{x\in B_y(h)\right\}\rho_b^2(x)\d x}_{=:\mathbf{I}(u)}\\
&\hspace*{10em}+\underbrace{\frac{(1-F_b(u))^2}{\rho_b^2(u)\sigma^4(u)}\int_{-\infty}^u\mathds{1}\left\{x\in B_y(h)\right\}\rho_b^2(x)\d x}_{=:\mathbf{II}(u)}\Bigg\}.
\end{align*}
Now, for $u>A$,
\begin{align*}
|\mathbf{I}(u)|&\leq\ \int_u^\infty\mathds{1}\left\{x\in B_y(h)\right\}\sigma^{-2}(u)\exp\left(4\int_u^x\frac{b(v)}{\sigma^2(v)}\d v\right)\d x\\
&\leq\ \int_u^\infty\mathds{1}\left\{x\in B_y(h)\right\}\sigma^{-2}(u)\exp\left(-4\gamma(x-u)\right)\d x\ \leq\ \underline\nu^{-2}\cdot 2h,\\
|\mathbf{II}(u)|&\leq\ M^2\cdot 2h\cdot\LL.
\end{align*}
The case $u<-A$ is treated analogously. 
Finally, for $u\in[-A,A]$, it holds
\begin{align*}
|g_y(u)|^2& \leq\ \left\|Q\left(\frac{y-\cdot}{h}\right)\right\|_{L^2(\R)}^2\sup_{-A\leq x\leq A}\frac{4\LL\cdot 2h}{\rho_b^2(x)}\\
&\leq\ 8\LL hC_{b,\sigma}^2\e^{2\C(2A+A^2)}\left\|Q\left(\frac{y-\cdot}{h}\right)\right\|_{L^2(\R)}^2.
\end{align*}
Consequently, applying the Burkholder--Davis--Gundy inequality, one obtains
\begin{align*}
\E_b\left[\left(\frac{1}{\sqrt t}|\Ma_t^y|\right)\right]&\le\ \Pi_1
 \E_b\left[\left(4t^{-1}\int_0^tg_y^2(X_s)\sigma^2(X_s)\d s\right)^{1/2}\right]\\
&\le\ \Pi_2\sqrt h\left\|Q\left(\frac{y-\cdot}{h}\right)\right\|_{L^2(\R)}\ \le\ \Pi_3h,
\end{align*}
for positive constants $\Pi_1,\Pi_2,\Pi_3$, {not depending on $y$.}
Thus,
\begin{align*}
\E_b\left[\int|\hat\rho_{t,h}(y)-\E_b\hat\rho_{t,h}(y)|\d y\right]&\le\
\frac{1}{th}\ \E_b\left[\int_{\com}\left(|\rd_t^y|+|\Ma_t^y|\right)\d y\right]\\
&\le\ \frac{\lebesgue(\com)}{th}\left(\Pi_0h+\Pi_3\sqrt t h\right)\ \le\ \frac{\nu}{\sqrt t}.
\end{align*}
For the bias, we obtain by standard arguments, exploiting H\"older continuity of $\rho_b$ and the order of the kernel $Q$,
\begin{align*}
\left|\E_b\hat \rho_{t,h}(y)-\rho_b(y)\right|&=\ \left|\frac{1}{h}\int Q\left(\frac{y-x}{h}\right)\left(\rho_b(x)-\rho_b(y)\right)\d x\right|\\
&\le\ \frac{h^\beta\mathcal L}{\lfloor\beta\rfloor!}\int|u^\beta Q(u)|\d u.
\end{align*}
The assertion now follows. 
\end{proof}

\end{appendix}

\bibliographystyle{abbrvnat}
\bibliography{impulse_nonparametric_aap_rev_arxiv}

\def\cprime{$'$} \def\cprime{$'$}
\begin{thebibliography}{27}
\providecommand{\natexlab}[1]{#1}
\providecommand{\url}[1]{\texttt{#1}}
\expandafter\ifx\csname urlstyle\endcsname\relax
  \providecommand{\doi}[1]{doi: #1}\else
  \providecommand{\doi}{doi: \begingroup \urlstyle{rm}\Url}\fi

\bibitem[Aeckerle-Willems and Strauch(2021)]{cacs18}
C.~Aeckerle-Willems and C.~Strauch.
\newblock Concentration of scalar ergodic diffusions and some statistical
  implications.
\newblock \emph{Annales de l'Institut Henri Poincar{\'e}, Probabilit{\'e}s et
  Statistiques}, 57\penalty0 (4):\penalty0 1857--1887, 2021.

\bibitem[Alvarez(2004{\natexlab{a}})]{alvarez2004class}
L.~H. Alvarez.
\newblock A class of solvable impulse control problems.
\newblock \emph{Applied Mathematics and Optimization}, 49\penalty0
  (3):\penalty0 265--295, 2004{\natexlab{a}}.

\bibitem[Alvarez(2004{\natexlab{b}})]{A04}
L.~H.~R. Alvarez.
\newblock Stochastic forest stand value and optimal timber harvesting.
\newblock \emph{SIAM J. Control Optim.}, 42\penalty0 (6):\penalty0 1972--1993
  (electronic), 2004{\natexlab{b}}.

\bibitem[Auer et~al.(2007)Auer, Ortner, and Szepesv{\'a}ri]{auer2007improved}
P.~Auer, R.~Ortner, and C.~Szepesv{\'a}ri.
\newblock Improved rates for the stochastic continuum-armed bandit problem.
\newblock In \emph{International Conference on Computational Learning Theory},
  pages 454--468. Springer, 2007.

\bibitem[Bayraktar and Zhang(2015)]{bayraktar2015minimizing}
E.~Bayraktar and Y.~Zhang.
\newblock Minimizing the probability of lifetime ruin under ambiguity aversion.
\newblock \emph{SIAM Journal on Control and Optimization}, 53\penalty0
  (1):\penalty0 58--90, 2015.

\bibitem[Bertsekas(2012)]{MR3642732}
D.~P. Bertsekas.
\newblock \emph{Dynamic programming and optimal control. {V}ol. {II}.
  {A}pproximate dynamic programming}.
\newblock Athena Scientific, Belmont, MA, fourth edition, 2012.

\bibitem[Bertsekas(2019)]{bertsekas_reinf}
D.~P. Bertsekas.
\newblock \emph{Reinforcement {L}earning and {O}ptimal {C}ontrol}.
\newblock Athena Scientific, Belmont, MA, 2019.

\bibitem[Bielecki et~al.(2019)Bielecki, Chen, Cialenco, Cousin, and
  Jeanblanc]{bielecki2019adaptive}
T.~R. Bielecki, T.~Chen, I.~Cialenco, A.~Cousin, and M.~Jeanblanc.
\newblock Adaptive robust control under model uncertainty.
\newblock \emph{SIAM Journal on Control and Optimization}, 57\penalty0
  (2):\penalty0 925--946, 2019.

\bibitem[Borodin and Salminen(2015)]{borodin-salminen}
A.~Borodin and P.~Salminen.
\newblock \emph{Handbook of {B}rownian Motion -- Facts and Formulae, 2nd
  edition, corrected printing}.
\newblock Birkh\"auser, Basel, Boston, Berlin, 2015.

\bibitem[Bubeck and Cesa-Bianchi(2012)]{bubeck2012regret}
S.~Bubeck and N.~Cesa-Bianchi.
\newblock Regret analysis of stochastic and nonstochastic multi-armed bandit
  problems.
\newblock \emph{Foundations and Trends{\textregistered} in Machine Learning},
  5\penalty0 (1):\penalty0 1--122, 2012.

\bibitem[Bubeck et~al.(2011)Bubeck, Munos, Stoltz, and
  Szepesv{\'a}ri]{bubeck2011x}
S.~Bubeck, R.~Munos, G.~Stoltz, and C.~Szepesv{\'a}ri.
\newblock X-armed bandits.
\newblock \emph{Journal of Machine Learning Research}, 12\penalty0
  (May):\penalty0 1655--1695, 2011.

\bibitem[Christensen(2013)]{christensen2013optimal}
S.~Christensen.
\newblock Optimal decision under ambiguity for diffusion processes.
\newblock \emph{Mathematical Methods of Operations Research}, 77\penalty0
  (2):\penalty0 207--226, 2013.

\bibitem[Christensen et~al.(2021{\natexlab{a}})Christensen, Neumann, and
  Sohr]{christensen2021competition}
S.~Christensen, B.~A. Neumann, and T.~Sohr.
\newblock Competition versus cooperation: A class of solvable mean field
  impulse control problems.
\newblock \emph{SIAM Journal on Control and Optimization}, 59\penalty0
  (5):\penalty0 3946--3972, 2021{\natexlab{a}}.

\bibitem[Christensen et~al.(2021{\natexlab{b}})Christensen, Strauch, and
  Trottner]{christensen2021learning}
S.~Christensen, C.~Strauch, and L.~Trottner.
\newblock Learning to reflect: A unifying approach for data-driven stochastic
  control strategies.
\newblock \emph{arXiv preprint arXiv:2104.11496}, 2021{\natexlab{b}}.

\bibitem[Gobet et~al.(2004)Gobet, Hoffmann, and Rei{\ss}]{goetal04}
E.~Gobet, M.~Hoffmann, and M.~Rei{\ss}.
\newblock Nonparametric estimation of scalar diffusions based on low frequency
  data.
\newblock \emph{Ann. Statist.}, 32\penalty0 (5):\penalty0 2223--2253, 2004.
\newblock ISSN 0090-5364.

\bibitem[Helmes et~al.(2017)Helmes, Stockbridge, and Zhu]{helmes2017continuous}
K.~L. Helmes, R.~H. Stockbridge, and C.~Zhu.
\newblock Continuous inventory models of diffusion type: long-term average cost
  criterion.
\newblock \emph{The Annals of Applied Probability}, 27\penalty0 (3):\penalty0
  1831--1885, 2017.

\bibitem[Helmes et~al.(2018)Helmes, Stockbridge, and Zhu]{helmes2018weak}
K.~L. Helmes, R.~H. Stockbridge, and C.~Zhu.
\newblock A weak convergence approach to inventory control using a long-term
  average criterion.
\newblock \emph{Advances in Applied Probability}, 50\penalty0 (4):\penalty0
  1032--1074, 2018.

\bibitem[Jobj{\"o}rnsson and Christensen(2018)]{jobjornsson2018anscombe}
S.~Jobj{\"o}rnsson and S.~Christensen.
\newblock Anscombe's model for sequential clinical trials revisited.
\newblock \emph{Sequential Analysis}, 37\penalty0 (1):\penalty0 115--144, 2018.

\bibitem[Kohler and Walk(2013)]{kohler2013data}
M.~Kohler and H.~Walk.
\newblock On data-based optimal stopping under stationarity and ergodicity.
\newblock \emph{Bernoulli}, 19\penalty0 (3):\penalty0 931--953, 2013.

\bibitem[Kutoyants(2004)]{kut04}
Y.~A. Kutoyants.
\newblock \emph{Statistical Inference for Ergodic Diffusion Processes}.
\newblock Springer Series in Statistics. Springer, New York, 2004.

\bibitem[Lattimore and Szepesv{\'a}ri(2020)]{Lattimore_bandits}
T.~Lattimore and C.~Szepesv{\'a}ri.
\newblock \emph{Bandit algorithms}.
\newblock Cambridge University Press, 2020.

\bibitem[Le~Gall(2016)]{legall16}
J.-F. Le~Gall.
\newblock \emph{Brownian motion, martingales, and stochastic calculus}, volume
  274 of \emph{Graduate Texts in Mathematics}.
\newblock Springer, 2016.
\newblock ISBN 978-3-319-31088-6; 978-3-319-31089-3.
\newblock \doi{10.1007/978-3-319-31089-3}.
\newblock URL \url{https://doi.org/10.1007/978-3-319-31089-3}.

\bibitem[Locatelli and Carpentier(2018)]{locatelli2018adaptivity}
A.~Locatelli and A.~Carpentier.
\newblock Adaptivity to smoothness in x-armed bandits.
\newblock In \emph{Conference on Learning Theory}, pages 1463--1492, 2018.

\bibitem[Pe{\v{s}}kir and Shiryaev(2006)]{ps}
G.~Pe{\v{s}}kir and A.~N. Shiryaev.
\newblock \emph{Optimal stopping and free-boundary problems}.
\newblock Lectures in Mathematics ETH Z\"urich. Birkh\"auser Verlag, Basel,
  2006.

\bibitem[Riedel(2009)]{R}
F.~Riedel.
\newblock Optimal stopping with multiple priors.
\newblock \emph{Econometrica}, 77\penalty0 (3):\penalty0 857--908, 2009.

\bibitem[Rieder and B{\"a}uerle(2005)]{rieder2005portfolio}
U.~Rieder and N.~B{\"a}uerle.
\newblock Portfolio optimization with unobservable markov-modulated drift
  process.
\newblock \emph{Journal of Applied Probability}, 42\penalty0 (2):\penalty0
  362--378, 2005.

\bibitem[Stettner(1986)]{stettner1986continuous}
{\L}.~Stettner.
\newblock On continuous time adaptive impulsive control.
\newblock In \emph{System Modelling and Optimization: Proceedings of 12th IFIP
  Conference, Budapest, Hungary, September 2--6, 1985}, pages 913--922.
  Springer, 1986.

\end{thebibliography}

\end{document}